\documentclass[11pt]{article} 

\usepackage{amsmath,amssymb,amsthm}
\usepackage{amssymb}
\usepackage{graphicx,subfigure,float,url}
\usepackage[english]{babel}
\usepackage[utf8]{inputenc}
\usepackage{enumitem}
\usepackage{fancybox}
\usepackage{listings}
\usepackage{mathtools}
\usepackage{amsfonts}

\usepackage{verbatim}

\usepackage{float}
\usepackage{makeidx}
\normalbaselines

\newtheorem{Theorem}{Theorem}[section]

\newtheorem{Lemma}[Theorem]{Lemma}

\newtheorem{Corollary}[Theorem]{Corollary}

\newtheorem{Remark}[Theorem]{Remark}

\newcommand{\RR}{{{\rm I} \kern -.15em {\rm R} }}

\newcommand{\C}{{{\rm l} \kern -.42em {\rm C} }}

\newcommand{\nat}{{{\rm I} \kern -.15em {\rm N} }}

\newcommand{\dv}{\mbox{\rm div}\ }

\newcommand{\be}{\begin{equation}}
\newcommand{\ee}{\end{equation}}
\newcommand{\beq}{\begin{eqnarray}}
\newcommand{\eeq}{\end{eqnarray}}
\newcommand{\beqs}{\begin{eqnarray*}}
\newcommand{\eeqs}{\end{eqnarray*}}
\newcommand{\bt}{\begin{Theorem}}
\newcommand{\et}{\end{Theorem}}
\newcommand{\br}{\begin{Remark}}
\newcommand{\er}{\end{Remark}}
\newcommand{\bc}{\begin{Corollary}}
\newcommand{\ec}{\end{Corollary}}
\newcommand{\bl}{\begin{Lemma}}
\newcommand{\el}{\end{Lemma}}
\newcommand{\bd}{\begin{definition}}
\newcommand{\ed}{\end{definition}}

\renewcommand{\geq}{\geqslant}
\renewcommand{\ge}{\geqslant}
\renewcommand{\leq}{\leqslant}
\renewcommand{\le}{\leqslant}

\title{Exponential decay of solutions to linear evolution  equations\\ with time-dependent time delay}
\author{
Elisa Continelli\footnote{Email: elisa.continelli@graduate.univaq.it}\hspace{0.3cm}\&\hspace{0.2cm}Cristina Pignotti\footnote{Email: cristina.pignotti@univaq.it} \\Dipartimento di Ingegneria e Scienze dell'Informazione e Matematica\\
		Universit\`{a} degli Studi di L'Aquila\\
		Via Vetoio, Loc. Coppito, 67100 L'Aquila Italy}

\date{}

\begin{document}

\textwidth=160 mm

\textheight=225mm

\parindent=8mm

\frenchspacing

\maketitle

\begin{abstract}
In this note, we analyze an abstract evolution equation with time-dependent time delay and time-dependent delay feedback coefficient. We assume that the operator corresponding to the nondelayed part of the model generates an exponentially stable semigroup. Under an appropriate assumption on the delay feedback, we prove the well-posedness and an exponential stability estimate for our model.  Applications of our abstract results to concrete models are also illustrated. 
\end{abstract}

\vspace{5 mm}

\def\qed{\hbox{\hskip 6pt\vrule width6pt
height7pt
depth1pt  \hskip1pt}\bigskip}

 %%{\bf 2020 Mathematics Subject Classification:} 93D15,


\section{Introduction}
The study of evolution equations in presence of delay or memory terms attracted, in recent years, the interest of many researchers. The presence of a time delay makes the problems more difficult to deal with and, of course, it is important to include in the models time delays/memory terms to take into account time lags, such as reaction times, maturation times, times needed to receive some information, etc., commonly present in real life phenomena.

On the other hand, it is well-known that a time delay can produce instability phenomena. In particular, for the damped wave equation, it has been proven that an arbitrarily small delay can make the model unstable even if it is  uniformly asymptotically stable in absence of delay effects (see e.g.  \cite{Datko, DLP, 
NPSicon06, XYL}).

 Nevertheless, suitable feedback laws can ensure the delayed model has the same stability properties as the  undelayed one (see \cite{NPSicon06, XYL}).
Here, we are interested in proving exponential stability estimates for abstract linear evolution equations in presence of a time delay feedback. In particular, we will consider time-variable time delays, that is time delays that are functions of the time variable.   

Let $\tau:[0, +\infty)\rightarrow (0, +\infty),$ be the time delay function belonging to $C(0,+\infty)$.
We assume that
\begin{equation}\label{bound}
0\le \tau (t)\le \bar\tau, \quad \forall \ t\ge 0,
\end{equation}
for a suitable positive constant $\bar{\tau}$.
\\We consider the following abstract model:

\begin{equation}\label{PA}
\begin{array}{l}
\displaystyle{U'(t)=AU(t)+k(t)BU(t-\tau(t)),\quad t\in(0,\infty),}\\
\displaystyle{U(t)=f(t)\quad t\in [-\bar\tau , 0],}
\end{array}
\end{equation}
where the operator $A$ generates an exponentially stable $C_0$-semigroup $(S(t))_{t\ge 0}$ in a Hilbert space $H$, and $B$ is a  continuous linear operator of $H$ into itself. Moreover, the delay damping coefficient $k:[-\bar{\tau},+\infty)\to\RR$ belongs to ${\mathcal L}^1_{loc}([-\bar{\tau},+\infty);\RR )$, and $\tau$ is a continuous time delay function satisfying \eqref{bound}. We denote with $U_0:=f(0)$.
\\By the assumptions on the operator $A,$ there exist two positive constants $M$ and $\omega$ such that 
\begin{equation}\label{SGdecay}
\Vert S(t)\Vert \le M e^{-\omega t}\,.
\end{equation}
Under some mild assumptions on the involved functions and parameters, we will establish the  well-posedness of the problem \eqref{PA}, and we will obtain exponential decay estimates for its solutions.

Moreover, on the delay feedback coefficient, we assume that the integrals on intervals of length $\bar\tau$ are uniformly bounded, namely, 
\begin{equation}\label{damp_coeff}
\int_{t-\bar{\tau}}^t |k(s)|ds\leq K, \quad \forall \ t\ge  0,
\end{equation}
for some $K>0.$

Stability results for abstract evolution equations with delay have been already studied in   \cite{JEE15, NicaisePignotti18, KP}. In \cite{JEE15, NicaisePignotti18} it is analyzed the case of a single constant delay and also the delay damping coefficient is assumed to be constant. The analysis has then been extended in \cite{KP} by considering, as here, (multiple) time dependent time delays.
However,  here we work in a more general setting. Indeed, in \cite{KP}, the classical set of assumptions usually employed to deal with wave-type equations in presence of time variable time delays (see e.g. \cite{NPV11, ChentoufMansouri, Feng} ) is used. In particular, in addition to \eqref{bound}, it is required that
$\tau \in W^{1,\infty}(0, +\infty)$ and that
$\tau^\prime (t)\le c<1.$ 
On the contrary, in the present note, we only assume that the time delay is a function bounded from above. Therefore, our results significantly
extend and improve the ones in \cite{KP}.

A concrete model that can be rewritten in the form \eqref{PA} is, e.g., the wave equation with frictional damping and delay feedback. In the case of constant delay feedback coefficient and constant time delay, this model has been first studied in \cite{SCL12}. Under a suitable smallness condition on the delay term coefficient, an exponential decay estimate has been proven. This result has then  been extended to linear  wave equations with internal delay feedback and boundary dissipative condition in  \cite{ANP10}. In \cite{AlNP, guesmia} it is instead analyzed the case of the wave equations with  delay feedback and viscoelastic damping. 

For other stability estimates in the  presence of  time delay effects, for specific models, mainly in the case of constant time delay and constant delay damping coefficient, we quote, among the others,  \cite{Ait, AG, AM, Chentouf, Dai, Oquendo, Said, Ma}. We mention also the recent papers  \cite{KP2} and \cite{Capistrano} dealing with delayed Korteweg-de Vries-Burgers and  higher-order dispersive equations,
respectively.

The rest of the paper is organized as follows. In section \ref{well} we prove the well-posedness of the model. We give two different proofs: the first, via a step by step procedure,  under the additional assumption that the time delay function is bounded from below by a positive constant; the second, in the general setting, using a fixed point theorem.   In section \ref{stab}, we deduce the exponential stability estimate for the linear model. In section \ref{nonlin}, we extend the well-posedness and exponential stability results to a nonlinear model with a Lipschitz continuous nonlinearity. Finally, section \ref{Examples} illustrates some concrete examples for which the abstract theory is applicable: namely the damped wave equation with delay and the elasticity system with delay.

\section{Well-posedness }\label{well}
 \hspace{5mm}

\setcounter{equation}{0}
Now, we prove the following well-posedness result.
\begin{Theorem}\label{well_p}
Let $f:[-\bar\tau, 0]\to H$ be a continuous function. Then, the problem \eqref{PA} has a unique (weak) solution given by Duhamel's formula
\begin{equation}\label{Du}
U(t)=S(t)U_0+\int_0^tS(t-s)\,k(s)BU(s-\tau(s))\ ds,
\end{equation}
for all $t\ge 0.$
\end{Theorem}
\begin{proof}
	Let $f\in C([-\bar\tau, 0];H)$. We give two different proofs, the first one only valid under the additional assumption that the time delay is bounded from below by a positive constant.
\\{\bf Case 1.} Assume that
\begin{equation}\label{below}
	\tau(t)\ge\tau_0>0, \quad \forall t\ge 0,
\end{equation}
for a suitable positive constant $\tau_0.$ We can  argue step by step, as in the proof of Proposition 2.1 of \cite{KP}, by restricting ourselves each time to time intervals of length $\tau_0.$
\\First we consider $t\in [0,\tau_0]$. Then, from \eqref{below}, $t-\tau(t)\in [-\bar{\tau}, 0].$  So,
setting $F(t) =k(t)BU(t-\tau(t)),$ $t\in [0,\tau_0],$ we have that $F(t)= k(t)Bf(t-\tau(t)),$ $t\in [0,\tau_0]$.
Then, problem \eqref{PA} can be rewritten, in the  interval $[0,\tau_0]$, as a standard inhomogeneous evolution problem:
\begin{equation}\label{step1}
\begin{array}{l}
U'(t)=AU(t)+F(t) \quad \mbox{in}\ (0,\tau_0),\\
U(0)=U_0.
\end{array}
\end{equation}
Since  $k\in \mathcal{L}^1_{loc}([-\bar{\tau},+\infty);\RR)$, $B$ is a bounded linear operator and $f\in C([-\bar\tau,0];H)$, we have that $F\in \mathcal{L}^1((0,\tau_0);H)$.
Therefore, applying  \cite[Corollary 2.2]{Pazy} there exists a unique solution $U \in C([0,\tau_0]; H)$ of \eqref{step1} satisfying the Duhamel's formula
$$
U(t)= S(t)U_0+\int_0^tS(t-s) F (s) ds,\quad t\in [0,\tau_0].
$$
Therefore,
$$
U(t)= S(t)U_0+\int_0^tS(t-s)k(s)BU(s-\tau(s)) ds,\quad t\in [0,\tau_0].
$$
Next, we consider the time interval $[\tau_0, 2\tau_0]$ and also define
$F(t)=k(t)BU(t-\tau(t)),$ for $t\in [\tau_0, 2\tau_0].$ Note that, if  $t\in [\tau_0, 2\tau_0],$ then $t-\tau(t)\in [-\bar{\tau}, \tau_0]$ and so $U(t-\tau(t))$ is known from the first step. 
Then $F_{\vert_{ [\tau_0, 2\tau_0]}}$ is a known function and it belongs to $\mathcal{L}^1((\tau_0,2\tau_0); H)$.
So we can rewrite our model \eqref{PA} in the time interval $[\tau_0,2\tau_0 ]$ as the inhomogeneous evolution problem
\begin{equation}\label{step2}
\begin{array}{l}
U'(t)=AU(t)+F(t) \quad \mbox {for }\ t\in (\tau_0 ,2\tau_0),\\
U(\tau)=U (\tau_0^-).
\end{array}
\end{equation}
Then, by the standard theory of abstract Cauchy problems, we have a unique continuous solution $U:[\tau_0, 2\tau_0]\rightarrow H$ satisfying
$$
U(t)= S(t-\tau_0 )U(\tau_0^-) +\int_{\tau_0} ^tS(t -s) F (s) ds,\quad t\in [\tau_0 ,2\tau_0],
$$
and so
$$
U(t)= S(t-\tau_0 )U(\tau_0^-) +\int_{\tau_0} ^t S(t -s) k(s)BU (s-\tau (s)) ds,\quad t\in [\tau_0 ,2\tau_0].
$$
Putting together the partial solutions obtained in the first and second steps we have a unique continuous solution $U:[0,2\tau_0]\rightarrow \RR$
satisfying  the Duhamel's formula
$$
U(t)= S(t)U_0+\int_0^t S(t-s) k(s)BU (s-\tau(s)) ds,\quad t\in [0,2\tau_0].
$$
Iterating this procedure we can find a unique solution $U\in C([0,+\infty);H)$ satisfying the representation formula \eqref{Du}.
\\{\bf Case 2.}	Let $\tau(\cdot)$ be a continuous function satisfying \eqref{bound}. In this case, since assumption \eqref{below} does not necessarily hold, we cannot use a step by step procedure in order to get the existence of a solution to \eqref{PA} with the initial datum $f$, as we did in {Case 1}. We have rather to employ a different method, based on the use of the Banach's fixed point Theorem.
\\Now, since $k\in {\mathcal L}^1_{loc}([-\bar{\tau},+\infty);\RR )$, there exists $T>0$ such that 
\begin{equation}\label{contrazione}
	\lVert k\rVert_{{\mathcal L}^1([0,T];\RR )}=\int_{0}^{T}\lvert k(s)\rvert ds<\frac{1}{M\lVert B\rVert}.
\end{equation} 
We define the set $$C_f([-\bar\tau,T];H):=\{U\in C([-\bar\tau,T];H): U(s)=f(s),\,\forall s\in [-\bar{\tau},0]\}.$$ 
Let us note that $C_f([-\bar\tau,T];H)\neq \emptyset$, since it suffices to take 
$$U(t)=\begin{cases}
	U_0,\quad &t\in(0,T],\\
	f(t),\quad &t\in [-\bar{\tau},0],
\end{cases}$$
to have that $U	\in C_f([-\bar\tau,T];H)$. 
\\It is immediate to see that 
$C_f([-\bar\tau,T];H)$ with the norm
$$\lVert U\rVert_{C([-\bar\tau,T];H)}=\max_{r\in[-\bar\tau, T]}\lVert U(r)\rVert,\quad \forall U\in C([-\bar\tau,T];H),$$
is a Banach space.
\\Now, we define the map
$\Gamma:C_f([-\bar{\tau},T];H)\rightarrow C_f([-\bar{\tau},T];H)$ given by
$$\Gamma U(t)=\begin{cases}
	S(t)U_0+\int_{0}^t S(t-s)k(s)BU(s-\tau(s))\, ds,\quad &t\in [0,T],
	\\f(t),\quad &t\in [-\bar{\tau},0).
\end{cases}$$
We claim that $\Gamma$ is well-defined. Indeed, let $U\in C_f([-\bar{\tau},T];H)$. Then, 
from the semigroup theory, $t\mapsto S(t)U_0$ is continuous. Also, since $U(\cdot)$ is continuous in $[-\bar{\tau},T]$, $\tau(\cdot)$ is a continuous function and $B$ is a bounded linear operator from $H$ into itself, $[0,T]\ni t\mapsto BU(t-\tau(t))$ is continuos. Moreover, $k\in {\mathcal L}^1([0,T];\RR )$. So $k(\cdot)BU(\cdot-\tau(\cdot))\in {\mathcal L}^1([0,T];H )$. Hence, the map $t\mapsto \int_0^t S(t-s)k(s)BU(s-\tau(s))ds$ is continuous in $[0,T]$. Thus, $\Gamma U\in C([0,T];H)$. Finally, since $\Gamma U=f$ in $[\bar{\tau},0]$ with $f\in C([-\bar{\tau},0];H)$ and $f(0)=U_0$, $\Gamma U\in C_f([\bar{\tau},T];H)$ and $\Gamma$ is well-defined.
\\Now, let $U,V\in C_f([\bar{\tau},T];H)$. For all $t\in [-\bar{\tau},0]$,
$$\lVert \Gamma U(t)-\Gamma V(t)\rVert=0.$$
On the other hand, for all $t\in (0,T]$,
$$\begin{array}{l}
	\vspace{0.3cm}\displaystyle{\lVert \Gamma U(t)-\Gamma V(t)\rVert\leq \int_{0}^{t}\lVert S(t-s)\rVert\lvert k(s)\rvert \lVert BU(s-\tau(s))-BV(s-\tau(s))\rVert ds}\\
	\vspace{0.3cm}\displaystyle{\hspace{2cm}\leq M\lVert B\rVert\int_{0}^{t}\lvert k(s)\rvert \lVert U(s-\tau(s))-V(s-\tau(s))\rVert ds}\\
	\vspace{0.3cm}\displaystyle{\hspace{2cm}\leq M\lVert B\rVert\lVert U-V\rVert_{C([-\bar{\tau},T];H)}\int_{0}^{T}\lvert k(s)\rvert  ds}\\
	\displaystyle{\hspace{2cm}=MB\lVert k\rVert_{\mathcal{L}^1([0,T];\RR)}\lVert U-V\rVert_{C([-\bar{\tau},T];H)}.}
\end{array}$$
Thus, $$\lVert \Gamma U(t)-\Gamma V(t)\rVert\leq MB\lVert k\rVert_{\mathcal{L}^1([0,T];\RR)}\lVert U-V\rVert_{C([-\bar{\tau},T];H)},\quad\forall t\in [-\bar{\tau},T],$$
from which $$\lVert \Gamma U-\Gamma V\rVert_{C([-\bar{\tau},T];H)}\leq M\lVert B\rVert\lVert k\rVert_{\mathcal{L}^1([0,T];\RR)}\lVert U-V\rVert_{C([-\bar{\tau},T];H)}.$$
Now, from \eqref{contrazione} we have that $M\lVert B\rVert\lVert k\rVert_{\mathcal{L}^1([0,T];\RR)}<1$. Hence, $\Gamma $ is a contraction. Then, from the Banach's Theorem, $\Gamma $ has a unique fixed point $U\in C_f([-\bar{\tau},T];H)$, i.e. \eqref{PA} has a unique solution $U\in C([0,T];H)$ given by the Duhamel's formula $$U(t)=S(t)U_0+\int_{0}^t S(t-s)k(s)BU(s-\tau(s))\,ds,\quad \forall t\in [0,T].$$
Now, let us note that the solution $U$ is bounded. Indeed, for all $t\in [0,T]$,
$$\begin{array}{l}
	\displaystyle{\lVert U(t)\rVert\leq M\lVert U_0\rVert+M\lVert B\rVert\int_{0}^{t}\lvert k(s)\rvert \lVert U(s-\tau(s))\rVert ds}\\
	\displaystyle{\hspace{1cm}\leq M\lVert U_0\rVert+M\lVert B\rVert\lVert k\rVert_{\mathcal{L}^1([0,T];\RR)}\max_{r\in[-\bar\tau, 0]}\lVert f(r)\rVert+M\lVert B\rVert\int_{0}^{t}\lvert k(s)\rvert \max_{r\in[0,s]}\lVert U(r)\rVert ds.}
\end{array}$$
Then, $$\begin{array}{l}
	\displaystyle{\max_{r\in[0,t]}\lVert U(r)\rVert\leq M\left(\lVert U_0\rVert+\lVert B\rVert\lVert k\rVert_{\mathcal{L}^1([0,T];\RR)}\max_{r\in[-\bar\tau, 0]}\lVert f(r)\rVert\right)+M\lVert B\rVert\int_{0}^{t}\lvert k(s)\rvert \max_{r\in[0,s]}\lVert U(r)\rVert ds.}
\end{array}$$
Hence,  the Gronwall's estimate yields
$$\max_{r\in[0,t]}\lVert U(r)\rVert\leq M\left(\lVert U_0\rVert+\lVert B\rVert\lVert k\rVert_{\mathcal{L}^1([0,T];\RR)}\max_{r\in[-\bar\tau, 0]}\lVert f(r)\rVert\right)e^{M\lVert B\rVert\int_{0}^{t}\lvert k(s)\rvert ds,}$$
from which, taking into account of \eqref{contrazione},
$$\lVert U(t)\rVert\leq e\left(M\lVert U_0\rVert+\max_{r\in[-\bar\tau, 0]}\lVert f(r)\rVert\right), \quad \forall t\in [0,T].$$
Thus, the solution $U$ is bounded and we can extend it up to some maximal interval $[0,\delta)$, $\delta>0$. We claim that $\delta=+\infty$. Indeed, assume by contradiction that $\delta<+\infty$. Then, being $U$ bounded, we can consider the following problem
\begin{equation}\label{PAsecondo}
	\begin{array}{l}
		\displaystyle{V'(t)=AV(t)+k(t)BV(t-\tau(t)),\quad t\in(0,\infty),}\\
		\displaystyle{V(t)=U(t)\quad t\in [\delta-\bar\tau , \delta),}\\
		\displaystyle{V(\delta)=U(\delta^-).}
\end{array}
\end{equation}
Arguing as before, there exists $T'>0$ such that
\begin{equation}\label{contrazione2}
	\lVert k\rVert_{\mathcal{L}^1([\delta,T'],\RR)}=\int_{\delta}^{T'}\lvert k(s)\rvert ds<\frac{1}{M\lVert B\rVert}.
\end{equation}
We then set
$$C_U([\delta-\bar{\tau},T'];H)=\{V\in C([\delta-\bar{\tau},T'];H):V(s)=U(s),\forall s\in [\delta-\bar{\tau},\delta), V(\delta)=U(\delta^-)\},$$
which is a nonempty closed subset of $C([\delta-\bar{\tau},T'];H)$. \\Next, we define the map $\Gamma:C_U([-\bar{\tau},T];H)\rightarrow C_U([-\bar{\tau},T];H)$ given by
$$\Gamma V(t)=\begin{cases}
	S(t-\delta)U(\delta^-)+\int_{\delta}^t S(t-s)k(s)BV(s-\tau(s)),\quad &t\in [\delta,T'],
	\\U(t),\quad &t\in [\delta-\bar{\tau},\delta).
\end{cases}$$
We have that $\Gamma$ is well-defined and, using the same arguments employed at the beginning of {Case 2}, \eqref{contrazione2} implies that $\Gamma$ is a contraction. So, $\Gamma$ has a unique fixed point, i.e. \eqref{PAsecondo} has a unique continuous $V$ solution given by the Duhamel formula $$V(t)=S(t-\delta)U(\delta^-)+\int_{\delta}^t S(t-s)k(s)BV(s-\tau(s)),\quad t\in [\delta,T'].$$
Thus, putting together the solutions $U$, $V$, we get the existence of a unique continuous solution to \eqref{PA} that satisfies the Duhamel's formula \eqref{Du} and that is defined in $[0,\delta')$, with $\delta'>\delta$. This contradicts the maximality of $\delta$. Hence, $\delta=+\infty$, i.e. \eqref{PA} has a unique global solution $U\in C([0,+\infty);H)$ given by \eqref{Du}.
\end{proof}

\section{Exponential stability }\label{stab}
 \hspace{5mm}

\setcounter{equation}{0}

Under an appropriate relation between the problem's parameters the system \eqref{PA} is exponentially stable.
In particular, we assume that

\begin{equation}\label{assumption_delay}
M  \Vert B\Vert e^{\omega\bar{\tau}} \int_0^t \vert k(s)\vert ds  \leq \gamma+\omega' t, \quad \forall  t\geq0,
\end{equation}
for suitable constants $\gamma\geq 0$ and $\omega'\in[0,\omega).$

\begin{Theorem}\label{generaleCV}
Assume \eqref{assumption_delay}.
For every $f\in C([-\bar\tau, 0]; H),$ the solution $U\in C([0, +\infty); H)$ to \eqref{PA} with the initial datum $f$ satisfies the exponential decay estimate
\begin{equation}
\label{stimaesponenziale}
||U(t)||\leq Me^\gamma \left (\Vert U_0\Vert+e^{\omega \bar\tau}K\lVert B\rVert \max_{s\in[-\bar\tau, 0]}\left\{\Vert e^{\omega s} f(s)\Vert \right\}\right )e^{-(\omega -\omega')t}, 
\end{equation}
for any $t\geq 0$.
\end{Theorem}
\begin{proof} Let $f\in C([-\bar\tau, 0]; H)$. Let $U\in C([0, +\infty); H)$ be the solution to \eqref{PA} with the initial condition $f$. From Duhamel's Formula, we have that
$$
\displaystyle{||U(t)||\leq Me^{-\omega t} ||U_0||+Me^{-\omega t} \int_0^t e^{\omega s}|k(s)| \cdot ||BU(s-\tau(s))|| ds, \quad \forall t\ge 0. }
$$
Then, for all $t\geq \bar{\tau}$, we deduce that
\begin{equation}\label{N1}
\begin{array}{l}
\displaystyle{ ||U(t)||\le Me^{-\omega t} ||U_0||+Me^{-\omega t} \int_0^{\bar\tau} e^{\omega s} |k(s)| \cdot ||{B}U(s-\tau(s))|| ds}\\
\hspace{1.5 cm}
\displaystyle{+ Me^{-\omega t} \int_{\bar\tau}^t e^{\omega s}  |k(s)|\cdot ||BU(s-\tau(s))||  ds }\\
\hspace{1.2 cm}
\displaystyle{
\le Me^{-\omega t} ||U_0||+M||B||e^{-\omega t}e^{\omega\bar\tau} \int_0^{\bar\tau} e^{\omega (s-\tau(s))} |k(s)| \cdot ||  U(s-\tau(s))|| ds}\\
\hspace{1.5 cm}
\displaystyle{+ M||B||e^{-\omega t}e^{\omega\bar\tau} \int_{\bar\tau}^t e^{\omega (s-\tau(s))} |k(s)|\cdot ||U(s-\tau(s))||  ds.}
\end{array}
\end{equation}
Now, observe that
\begin{equation}\label{March16_4}
\begin{array}{l}
\displaystyle{
\int_0^{\bar\tau} e^{\omega (s-\tau(s))}\vert k(s)\vert\cdot \Vert U(s-\tau(s))\Vert ds}\\
\displaystyle{\hspace{1 cm}
\le
\int_0^{\bar\tau} \vert k(s)\vert\left (   \max_{r\in[-\bar\tau, 0]}\left\{ e^{\omega r}\Vert f(r)\Vert\right \}+\max_{r\in[0,s]}  \left\{e^{\omega r}\Vert U(r)\Vert \right\}    \right) ds}\\
\hspace{1 cm}\le\displaystyle{ K \max_{r\in[-\bar\tau, 0]}\left\{e^{\omega r}\Vert f(r)\Vert\right\}+\int_0^{\bar\tau} \vert k(s)\vert \max_{r\in[0,s]}\left\{e^{\omega r}\Vert U(r)\Vert\right\} ds.
}
\end{array}
\end{equation}
Then, using \eqref{March16_4} in \eqref{N1}, we deduce
\begin{equation}\label{diseq}
	\begin{array}{l}
		\displaystyle{ ||U(t)||\le Me^{-\omega t} \left (||U_0||+
			e^{\omega\overline\tau}  K\Vert B\Vert \max_{r\in[-\bar\tau, 0]}\left\{e^{\omega r}\Vert f(r)\Vert\right\}
			\right )
		}\\
		\hspace{1,5 cm}
		\displaystyle{+ M\Vert B\Vert e^{-\omega t} e^{\omega\bar\tau}\int_{0}^t  |k(s)|\max_{r\in [s-\bar\tau, s]\cap [0,s]} \left\{e^{\omega r}||U(r)||\right\}  ds,}
	\end{array}
\end{equation}
for all $t\geq \bar{\tau}$. On the other hand, for all $t\in [0,\bar{\tau}]$, it holds that
$$
\begin{array}{l}
	\displaystyle{ ||U(t)||\le Me^{-\omega t} ||U_0||+M\Vert B\Vert e^{-\omega t}e^{\omega \bar{\tau}} \int_0^{t} e^{\omega (s-\tau(s))} |k(s)| \cdot ||U(s-\tau(s))|| ds.}
\end{array}
$$
Then, arguing as before,
$$\begin{array}{l}
	\displaystyle{
		\int_0^{t} e^{\omega (s-\tau(s))}\vert k(s)\vert\cdot \Vert U(s-\tau(s))\Vert ds}\\
	\displaystyle{\hspace{1 cm}
		\le
		\int_0^{t} \vert k(s)\vert\left (   \max_{r\in[-\overline\tau, 0]}\left\{ e^{\omega r}\Vert f(r)\Vert\right \}+\max_{r\in[0,s]}  \left\{e^{\omega r}\Vert U(r)\Vert \right\}    \right) ds}\\
	\displaystyle{\hspace{1 cm}\le  K\max_{r\in[-\bar\tau, 0]}\left\{e^{\omega r}\Vert f(r)\Vert\right\}+\int_0^{t}\vert k(s)\vert \max_{r\in[0,s]}\left\{e^{\omega r}\Vert U(r)\Vert\right\} ds}
\end{array}$$
for all $t\in [0,\bar{\tau}]$. Hence, \eqref{diseq} holds also for $t\in [0,\bar{\tau}]$. So, we deduce
$$
\begin{array}{l}
\displaystyle{ e^{\omega t}||U(t)|| \le M \left (||U_0||+
e^{\omega\bar\tau} K\Vert B\Vert \max_{s\in[-\bar\tau, 0]}\left\{\Vert e^{\omega s} f(s)\Vert\right\}
\right )}\\
\hspace{2 cm}
\displaystyle{+M  \Vert B\Vert  e^{\omega\bar\tau} \int_0^t\vert k(s)\vert  \max_{r\in [s-\bar\tau, s]\cap [0,s]}\left\{e^{\omega r} ||U(r)||\right \}ds, \quad \forall t\ge 0.}
\end{array}
$$
Now, we note that it also holds 
$$
\begin{array}{l}
\displaystyle{ \max_{s\in [t-\bar\tau, t]\cap [0,t]}\left\{e^{\omega s}||U(s)||\right\} \le M  \left (||U_0||+
e^{\omega\bar\tau} K\Vert B\Vert \max_{s\in [-\bar\tau, 0]}\left\{\Vert e^{\omega s} f(s)\Vert\right\}
\right )}\\
\hspace{2 cm}
\displaystyle{+M \Vert B\Vert e^{\omega\bar\tau}  \int_0^t  \vert k(s)\vert   \max_{r\in [s-\bar\tau, s]\cap [0,s]}\left\{e^{\omega r} ||U(r)||\right \}ds, \quad \forall t\ge 0.}
\end{array}
$$
Hence, if we denote 
$$\tilde u (t):=\max_{s\in [t-\bar\tau, t]\cap [0,t]}\left\{e^{\omega s}||U(s)||\right\},$$
 Gronwall's estimate implies
$${\tilde u}(t)\le {\tilde M} e^{  M\Vert B\Vert  e^{\omega\overline\tau}\int_0^t \vert k(s)\vert ds},\quad \forall t\geq 0,
$$
where
$$\tilde M:= M\left (
\Vert U_0\Vert+e^{\omega \bar\tau}K \Vert B\Vert \max_{s\in[-\bar\tau, 0]}\left\{\Vert e^{\omega s} f(s)\Vert\right\}
\right ). 
$$
Then, 
$$e^{\omega t}\Vert U(t)\Vert\le {\tilde M} e^{  M \Vert B\Vert e^{\omega\bar\tau}\int_0^t \vert k(s)\vert ds},\quad \forall t\geq 0.
$$
Finally, assumption \eqref{assumption_delay}, yields
$$\begin{array}{l}
	\vspace{0.3cm}\displaystyle{\Vert U(t)\Vert\le {\tilde M} e^{  M \Vert B\Vert e^{\omega\bar\tau}\int_0^t \vert k(s)\vert ds}e^{-\omega t}}\\
	\displaystyle{\hspace{2cm}\leq {\tilde M} e^{  \gamma+\omega't}e^{-\omega t}=\tilde{M}e^\gamma e^{-(\omega-\omega')t},}
\end{array}$$
for all $t\geq 0$, which proves the exponential decay estimate \eqref{stimaesponenziale}.
\end{proof}

\section{A nonlinear model}\label{nonlin}
 \hspace{5mm}

\setcounter{equation}{0}
As an easy generalization of previous results, we can consider the nonlinear model

\begin{equation}\label{PANL}
\begin{array}{l}
\displaystyle{U'(t)=AU(t)+k(t)BU(t-\tau(t))+G(U(t)),\quad t\in(0,\infty),}\\
\displaystyle{U(t)=f(t)\quad t\in [-\bar\tau , 0],}
\end{array}
\end{equation}
where $A,$$B,$ $k(\cdot), \tau(\cdot)$ are as before, and we denote $U_0:=f(0)$. Moreover, $G:H\rightarrow H$ is Lipschitz continuous, namely there exists $L>0$ such that
\begin{equation}\label{Lip}
\Vert G(U_1)-G(U_2)\Vert\le L \Vert U_1-U_2\Vert, \quad \forall\ U_1, U_2\in H,
\end{equation}
and we assume that $G(0)=0$.

Analogously to before, one can first give a well-posedness result. See \cite{NicaisePignotti18} for the proof in the case of constant time delay.

\begin{Theorem}\label{well_pNL}

Let $f:[-\bar \tau, 0]\to H$ be a continuous function. Then, the problem \eqref{PANL} has a unique (weak) solution given by Duhamel's formula
\begin{equation}\label{DuNL}
U(t)=S(t)U_0+\int_0^tS(t-s)[G(U(s))+ k(s)BU(s-\tau(s))]\ ds,
\end{equation}
for all $t\ge 0.$
\end{Theorem}
\begin{proof}
Let $f\in C([-\bar\tau, 0],H)$. As before, we can give two different proofs.
\\{\bf Case 1} Assume that \eqref{below} holds tue. We can  argue step by step, as before, by restricting ourselves each time to time intervals of length $\tau_0.$
\\First we consider $t\in [0,\tau_0]$. Then, from \eqref{below}, $t-\tau(t)\in [-\bar{\tau}, 0].$  So,
setting $F(t) =k(t)BU(t-\tau(t)),$ $t\in [0,\tau_0],$ we have that $F(t)= k(t)Bf(t-\tau(t)),$ $t\in [0,\tau_0]$.
Then, problem \eqref{PANL} can be rewritten, in the  interval $[0,\tau_0]$, as a standard inhomogeneous evolution problem:
\begin{equation}\label{step1NL}
\begin{array}{l}
U'(t)=AU(t)+G(U(t))+F(t) \quad \mbox{in}\ (0,\tau_0),\\
U(0)=U_0.
\end{array}
\end{equation}
Since  $k\in \mathcal{L}^1_{loc}([-\bar{\tau},+\infty);\RR)$ and $f\in C([-\bar\tau,0];H)$, then we have that $F\in \mathcal{L}^1((0,\tau_0);H)$.
Therefore, applying  the standard theory for nonlinear evolution equations (see e.g. \cite{Pazy}), there exists a unique solution $U \in C([0,\tau_0]; H)$ of \eqref{step1NL} satisfying the Duhamel's formula
$$
U(t)= S(t)U_0+\int_0^tS(t-s) [G(U(s))+ F (s)] ds,\quad t\in [0,\tau_0].
$$
Therefore,
$$
U(t)= S(t)U_0+\int_0^tS(t-s)[ G(U(s))+k(s)BU(s-\tau(s)) ] ds,\quad t\in [0,\tau_0].
$$
Next, we consider the time interval $[\tau_0, 2\tau_0]$ and also define
$F(t)=k(t)BU(t-\tau(t)),$ for $t\in [\tau_0, 2\tau_0].$ Note that, if  $t\in [\tau_0, 2\tau_0],$ then $t-\tau(t)\in [-\bar{\tau}, \tau_0]$ and so $U(t-\tau(t))$ is known from the first step. 
Then $F_{\vert_{ [\tau_0, 2\tau_0]}}$ is a known function and it belongs to $\mathcal{L}^1((\tau_0,2\tau_0); H)$.
So we can rewrite our model \eqref{PANL} in the time interval $[\tau_0,2\tau_0 ]$ as the inhomogeneous evolution problem
\begin{equation}\label{step2NL}
\begin{array}{l}
U'(t)=AU(t)+G(U(t))+F(t) \quad \mbox {for }\ t\in (\tau_0 ,2\tau_0),\\
U(\tau)=U (\tau_0^-).
\end{array}
\end{equation}
Then, we have a unique continuous solution $U:[\tau_0, 2\tau_0]\rightarrow H$ satisfying
$$
U(t)= S(t-\tau_0 )U(\tau_0^-) +\int_{\tau_0} ^tS(t -s) [G(U(s))+F (s)] ds,\quad t\in [\tau_0 ,2\tau_0],
$$
and so
$$
U(t)= S(t-\tau_0 )U(\tau_0^-) +\int_{\tau_0} ^t S(t -s) [G(U(s))+k(s)BU (s-\tau (s))] ds,\quad t\in [\tau_0 ,2\tau_0].
$$
Putting together the partial solutions obtained in the first and second steps we have a unique continuous solution $U:[0,2\tau_0)\rightarrow \RR$
satisfying  the Duhamel's formula
$$
U(t)= S(t)U_0+\int_0^t S(t-s) [G(U(s))+k(s)BU (s-\tau(s))] ds,\quad t\in [0,2\tau_0].
$$
Iterating this procedure we can find a unique solution $U\in C([0,+\infty);H)$ satisfying the representation formula \eqref{DuNL}.
\\{\bf Case 2} Let $\tau(\cdot)$ be a continuous function satisfying \eqref{bound}. 
\\Now, since $k\in {\mathcal L}^1_{loc}([-\bar{\tau},+\infty);\RR )$, there exists $T>0$ sufficiently small such that 
\begin{equation}\label{contrazione3}
	LT+\lVert B\rVert\lVert k\rVert_{{\mathcal L}^1([0,T];\RR )}=LT+\lVert B\rVert\int_{0}^{T}\lvert k(s)\rvert ds<\frac{1}{M},
\end{equation}
where $L$ is the Lipschitz constant in \eqref{Lip}.
We define the set $$\tilde C_f([-\bar\tau,T];H):=\{U\in C([-\bar\tau,T];H): U(s)=f(s),\,\forall s\in [-\bar{\tau},0]\}.$$ 
Let us note that $\tilde C_f([-\bar\tau,T];H)$ is a nonempty and closed subset of $C([-\bar\tau,T];H)$.
Hence, $(\tilde C_f([-\bar\tau,T];H),\lVert \cdot\rVert_{C([-\bar\tau,T];H)})$ is a Banach space.
\\Next, we define the map
$\tilde\Gamma:\tilde C_f([-\bar{\tau},T];H)\rightarrow \tilde C_f([-\bar{\tau},T];H)$ given by
$$\tilde\Gamma U(t)=\begin{cases}
	S(t)U_0+\int_{0}^t S(t-s)[G(U(s))+k(s)BU(s-\tau(s))]\, ds,\quad &t\in [0,T],
	\\f(t),\quad &t\in [-\bar{\tau},0).
\end{cases}$$
Let us note that $\tilde\Gamma$ is well-defined. 
\\Moreover, $\tilde\Gamma$ is a contraction. Indeed, let $U,V\in \tilde C_f([\bar{\tau},T];H)$. Then, for all $t\in [-\bar{\tau},0]$,
$$\lVert \tilde\Gamma U(t)-\tilde\Gamma V(t)\rVert=0.$$
On the other hand, for all $t\in (0,T]$, since $G$ is Lipschitz continuous we get
$$\begin{array}{l}
	\vspace{0.3cm}\displaystyle{\lVert \tilde\Gamma U(t)-\tilde\Gamma V(t)\rVert\leq \int_{0}^{t}\lVert S(t-s)\rVert G(U(s))-G(V(s))\rVert ds}\\
	\vspace{0.3cm}\displaystyle{\hspace{2cm}+ \int_{0}^{t}\lVert S(t-s)\rVert\lvert k(s)\rvert \lVert BU(s-\tau(s))-BV(s-\tau(s))\rVert ds}\\
	
	\displaystyle{\hspace{2cm}\leq M(LT+\lVert B\rVert\lVert k\rVert_{\mathcal{L}^1([0,T];\RR)})\lVert U-V\rVert_{C([-\bar{\tau},T];H)}.}
\end{array}$$
Thus, 
$$\lVert \tilde\Gamma U-\tilde\Gamma V\rVert_{C([-\bar{\tau},T];H)}\leq M(LT+\lVert B\rVert\lVert k\rVert_{\mathcal{L}^1([0,T];\RR)})\lVert U-V\rVert_{C([-\bar{\tau},T];H)}.$$
As a consequence, from \eqref{contrazione3} $\tilde \Gamma$ is a contraction. Thus, from the Banach's Theorem, $\tilde \Gamma$ has a unique fixed point $U\in \tilde C_f([-\bar{\tau},T];H)$, i.e. \eqref{PA} has a unique solution $U\in C([0,T];H)$ given by the Duhamel formula $$U(t)=S(t)U_0+\int_{0}^t S(t-s)[G(U(s))+k(s)BU(s-\tau(s))] \, ds,\quad \forall t\in [0,T].$$
Now, note that the solution $U$ is bounded. So, arguing as in {\bf Case 2} of Theorem \ref{well_p}, we can conclude that \eqref{PANL} has a unique continuous global solution that satisfies the Duhamel's formula \eqref{DuNL}.
\end{proof}

Under an appropriate relation between the problem's parameters, the system \eqref{PANL} is exponentially stable.

\begin{Theorem}\label{generaleCVNL}
Assume \eqref{assumption_delay} and $L<\frac {\omega-\omega'} M.$ Then, for every $f\in C([-\bar\tau, 0]; H),$ the solution $U\in C([0,+\infty);H)$ to \eqref{PANL} with the initial datum $f$ satisfies the exponential decay estimate
\begin{equation}
\label{stimaesponenzialeNL}
||U(t)||\leq Me^\gamma \left (\Vert U_0\Vert+e^{\omega \bar\tau}K\lVert B\rVert \max_{s\in[-\bar\tau, 0]}\left\{\Vert e^{\omega s} f(s)\Vert \right\}\right )e^{-(\omega -\omega'-ML)t}, 
\end{equation}
for any $t\geq 0$.
\end{Theorem}
\begin{proof}
Let $f\in C([-\bar\tau, 0]; H)$ and let $U$ be the unique global solution to \eqref{well_pNL} with initial datum $f$. Then, from Duhamel's formula \eqref{DuNL}, we have that
$$
\begin{array}{l}
	\displaystyle{||U(t)||\leq Me^{-\omega t} ||U_0||+Me^{-\omega t} \int_0^t e^{\omega s}||G(U(s))||ds}\\
	\displaystyle{\hspace{2cm}+ M\lVert B\rVert e^{-\omega t}\int_0^t e^{\omega s}|k(s)| \cdot ||U(s-\tau(s))||) ds,}
\end{array}
$$
for all $t\geq 0$. Now, using the same arguments employed in Theorem 
\ref{generaleCV}, we get
\begin{equation}\label{diseq2}
	\begin{array}{l}
		\displaystyle{\int_0^t e^{\omega s}|k(s)| \cdot ||U(s-\tau(s))|| \,ds\leq e^{\omega\bar\tau}\int_{0}^{\bar\tau}e^{\omega (s-\tau(s))}|k(s)| \cdot ||U(s-\tau(s))||\, ds}\\
		\displaystyle{\hspace{3cm}+e^{\omega\bar\tau}\int_{\bar\tau}^{t}e^{\omega( s-\tau(s))}|k(s)| \cdot ||U(s-\tau(s))||\, ds}\\
		\displaystyle{\hspace{2.5cm}\leq e^{\omega\bar\tau}K\max_{r\in [-\bar\tau, 0]}}\{||e^{\omega r}f(r)||\}+e^{\omega\bar\tau}\int_{0}^{t}|k(s)| \cdot\max_{r\in [s-\bar\tau, s]\cap [0,s]}\{e^{\omega r}||U(r)||\}ds,
	\end{array}
\end{equation}
for all $t\geq \bar{\tau}$. Also, for all $t\in [0,\bar{\tau}]$,
$$\begin{array}{l}
	\vspace{0.3cm}\displaystyle{\int_0^t e^{\omega s}|k(s)| \cdot ||U(s-\tau(s))|| \,ds\leq e^{\omega\bar\tau}\int_{0}^{t}e^{\omega (s-\tau(s))}|k(s)| \cdot ||U(s-\tau(s))||\, ds}\\
	\displaystyle{\hspace{1cm}\leq e^{\omega\bar\tau}K\max_{r\in [-\bar\tau, 0]}\{||e^{\omega r}f(r)||\}+e^{\omega\bar\tau}\int_{0}^{t}|k(s)| \cdot\max_{r\in [s-\bar\tau, s]\cap [0,s]}\{e^{\omega r}||U(r)||\}ds,}
\end{array}$$
i.e. \eqref{diseq2} holds true for all $t\geq 0$. Hence, from \eqref{diseq2} we get
$$\displaystyle{ ||U(t)||\le Me^{-\omega t}\left(||U_0||+e^{\omega\bar{\tau}}K\lVert B\rVert\max_{r\in [-\bar\tau, 0]}\{||e^{\omega r}f(r)||\}\right)}\hspace{5 cm}$$
\vspace{-0.3cm}
$$\hspace{1 cm}\displaystyle{+Me^{-\omega t}e^{\omega\bar\tau}||B||\int_{0}^{t}|k(s)| \cdot\max_{r\in [s-\bar\tau, s]\cap [0,s]}\{e^{\omega r}||U(r)||\}ds+Me^{-\omega t}}\int_{0}^{t}e^{\omega s}||G(U(s))||ds.$$
Now, let us note that, being $G(0)=0$ and being $G$ Lipschitz continuous of constant $L$,
$$\displaystyle{\int_{0}^{t}e^{\omega s}||G(U(s))||ds\leq L\int_{0}^{t}e^{\omega s}||U(s)||\leq L\int_{0}^{t}\max_{r\in [s-\bar\tau, s]\cap [0,s]}\{e^{\omega r}||U(r)||\}ds.}$$
As a consequence, we can write
$$\begin{array}{l}
\displaystyle{ ||U(t)||\le Me^{-\omega t}(||U_0||+e^{\omega\bar{\tau}}K||B||\max_{r\in [-\bar\tau, 0]}\{||e^{\omega r}f(r)||\})}\\
\hspace{0,5 cm}\displaystyle{\hspace{2cm}+Me^{-\omega t}\int_{0}^{t}(e^{\omega\bar\tau}||B|||k(s)|+L)\cdot\max_{r\in [s-\bar\tau, s]\cap [0,s]}\{e^{\omega r}||U(r)||\}ds.}
\end{array}
$$
Then, we have that
$$\begin{array}{l}
\displaystyle{ e^{\omega t}||U(t)|| \le M \left (||U_0||+
e^{\omega\bar\tau} K ||B||\max_{r\in[-\bar\tau, 0]}\left\{\Vert e^{\omega r} f(r)\Vert\right\}
\right )}\\
\hspace{2 cm}
\displaystyle{+M  \int_0^t(e^{\omega\bar\tau}||B|||k(s)|+L)\cdot \max_{r\in [s-\bar\tau, s]\cap [0,s]}\left\{e^{\omega r} ||U(r)||\right \}ds, \quad \forall t\ge 0.}
\end{array}
$$
Now, we note that
$$
\begin{array}{l}
\displaystyle{ \max_{r\in [t-\bar\tau, t]\cap [0,t]}\left\{e^{\omega r}||U(r)||\right\} \le M  \left (||U_0||+
e^{\omega\bar\tau} K ||B||\max_{s\in [-\bar\tau, 0]}\left\{\Vert e^{\omega s} f(s)\Vert\right\}
\right )}\\
\hspace{2 cm}
\displaystyle{+M  \int_0^t (e^{\omega\bar\tau}||B|||k(s)|+L)\cdot \max_{r\in [s-\bar\tau, s]\cap [0,s]}\left\{e^{\omega r} ||U(r)||\right \}ds, \quad \quad \forall t\ge 0.}
\end{array}
$$
Hence, if we denote with
$$\tilde u (t):=\max_{r\in [t-\bar\tau, t]\cap [0,t]}\left\{e^{\omega r}||U(r)||\right\},$$
Gronwall's estimate yields
$${\tilde u}(t)\le {\tilde M} e^{  M\Vert B\Vert  e^{\omega\bar\tau}\int_0^t \vert k(s)\vert ds+MLt},\quad \forall t\ge 0
$$
where
$$\tilde M:= M\left (
\Vert U_0\Vert+e^{\omega \bar\tau}K||B|| \max_{r\in[-\bar\tau, 0]}\left\{\Vert e^{\omega r} f(r)\Vert\right\}
\right ). 
$$
Then, 
$$e^{\omega t}\Vert U(t)\Vert\le {\tilde M} e^{  M \Vert B\Vert e^{\omega\bar\tau}\int_0^t \vert k(s)\vert ds+MLt}.
$$
Finally, by the assumption \eqref{assumption_delay} and the assumption on the Lipschitz constant $L$, we get the exponential  decay estimate \eqref{stimaesponenzialeNL}.\end{proof}

\section{Examples}
\label{Examples}\hspace{5mm}
In this section we consider, as concrete examples for which previous abstract well-posedness and stability results hold, the wave equation with localized frictional damping and delay feedback and an elasticity system with analogous feedback laws.

\setcounter{equation}{0}
\subsection{The damped wave equation}

Let $\Omega$ be an open bounded subset of $\RR^d$, with boundary $\partial\Omega$ of class $C^2,$ and let $\mathcal O \subset \Omega$ be an open subset which satisfies the geometrical control property in \cite{Bardos}. 
For instance, $\mathcal O \subset \Omega$ can be a neighborhood of  the whole boundary $\partial\Omega$ or, denoting by $m$ the standard multiplier $m(x)=x-x_0,$ $x_0\in \RR^d,$ as in \cite{Lions}, $\mathcal O$ can be the intersection of $\Omega$ with an open neighborhood of the set 
$$\Gamma_0=\left\{ \,x\in\Gamma\ :\ m(x)\cdot \nu(x)>0\,\right\}.$$
Moreover, let $\tilde{\mathcal{O}}\subset\Omega$ be another open subset. Denoting by $\chi_ {\mathcal O}$ and $\chi_ {\tilde {\mathcal O}}$  the characteristic functions of the sets ${\mathcal O}$ and $\tilde{{\mathcal O}}$ respectively,  we consider the following wave equation
\begin{equation}\label{wave}
\begin{array}{l}
\displaystyle{u_{tt}(x,t)-\Delta u(x,t)+a\chi_{\mathcal{O}}(x) u_t(x,t)}\\
\hspace{3,8 cm}
\displaystyle{+k(t)\chi_{\tilde{\mathcal O}}(x)u_t(x,t-\tau(t))=0, \quad (x,t)\in\Omega\times (0,+\infty),}\\
\displaystyle{u(x,t)=0, \quad (x,t)\in\partial\Omega\times (0,+\infty),}\\
\displaystyle{u(x,s)=u_0(x,s), \quad u_t(x,s)= u_1(x,s), \quad (x,s)\in {\Omega}\times [-\bar\tau,0],}
\end{array}
\end{equation}
where $a$ is a positive constant, $\tau(t)$ is the time delay function satisfying $0\le\tau(t)\le\bar\tau,$  and the delayed damping coefficient $k(\cdot):[-\bar{\tau},+\infty)\to (0,+\infty)$ is a $\mathcal{L}^1_{loc}([-\bar{\tau},+\infty))$ function satisfying \eqref{damp_coeff}. 
%%$H=L^2(\Omega),$ and the operator $A: {\mathcal D}(A)\rightarrow H$ defined as
%%$A=-\Delta$ with domain ${\mathcal D}(A)=H^2(\Omega)\cap H_0^1(\Omega)$. 
%%The operator $A$ is positive, self-adjoint, with dense domain and compact inverse in $H.$
%%We then define $W_1:=L^2({\mathcal O}),  W_2:=L^2(\tilde{\mathcal O}),$ and the operators 
%%$$C:W_1\rightarrow H: \quad v\rightarrow \sqrt{a} \tilde v\chi_ {\mathcal O},$$
%%$$B:W_2\rightarrow H: \quad v\rightarrow \tilde v\chi_ {\tilde{\mathcal O}},$$
%where $\tilde v$
%is the extension of v by zero outside  ${\mathcal O}$ and ${\tilde{\mathcal O}}$ %respectively.
Denoting $v(t)=u_t(t)$ and $U(t)=(u(t),v(t))^T,$ for any $t\ge 0,$, we can rewrite system \eqref{wave} in the abstract form \eqref{PA}, with ${\mathcal H}=H_0^1(\Omega)\times L^2(\Omega)$,
$$
 A=\begin{pmatrix}
0 & Id \\
\Delta & -a\chi_{\mathcal O}
\end{pmatrix}
$$
and 
$$
 B \begin{pmatrix} u \\ v \end{pmatrix} = \begin{pmatrix} 0 \\ -\chi_{\tilde{\mathcal{O}}} v \end{pmatrix},  \quad \forall\ t \geq 0.
$$
We know that ${A}$ generates an exponentially stable $C_0$-semigroup $\{S(t)\}_{t\geq 0}$ (see e.g. \cite{K}), namely there exist $\omega,M>0$ such that
$$
||S(t)||_{\mathcal{L}(\mathcal H)} \leq Me^{-\omega t}, \quad \forall \ t\geq 0.
$$
Hence, under the assumption \eqref{assumption_delay}, the stability estimate of Theorem \ref{generaleCV} holds for such a model. Then, we can deduce an exponential decay estimate for the
energy functional
$$E(t):= \frac 12 \int_\Omega |u_t(x,t)|^2 dx+\frac 12 \int_{\Omega} |\nabla u(x,t)|^2 dx+\frac 12 \int_{t-\bar\tau}^t\int_{\tilde{\mathcal O}} |k(s)|\cdot |u_t(x,s)|^2 dx ds.
$$
\begin{Theorem}\label{exp_wave}
Under assumption \eqref{assumption_delay}, for all initial data $(u_0, u_1)\in C([-\bar\tau, 0]; H_0^1(\Omega)\times L^2(\Omega)),$ the solution to
\eqref{wave} satisfies the energy decay estimate
\begin{equation}\label{energydecay}
E(t) \le C_* e^{-\beta t}, \quad t\ge 0,
\end{equation}
where $C_*$ is a constant depending on the initial data and $\beta>0.$
\end{Theorem}

\begin{proof}
From the energy's definition,
\begin{equation}\label{P11}
\begin{array}{l}
\displaystyle{E(t)= \frac 1 2 \Vert U(t)\Vert^2+\frac 12 \int_{t-\overline\tau}^t\int_{\tilde{\mathcal O}} |k(s)|\cdot |u_t(x,s)|^2 dx ds}\\
\vspace{0.3 cm}\displaystyle{\hspace{2,3 cm}\le \frac 12 \Vert U(t)\Vert^2+\frac 12 \int_{t-\bar\tau}^t |k(s)|\Vert U(s)\Vert^2 ds.}
\end{array}
\end{equation}
Then, from Theorem \ref{generaleCV}, 
$$\Vert U(t)\Vert \le C_0e^{-(\omega-\omega^\prime)t},\quad \forall t\ge 0,$$
for a suitable constant $C_0$ depending on the initial data. 
So, we can estimate
$$\int_{t-\bar\tau}^t |k(s)|\Vert U(s)\Vert^2 ds \le 
C_0 K e^{(\omega-\omega^\prime)\bar\tau}\, e^{-(\omega-\omega^\prime)t}, \quad \forall t\ge 0.$$
By using the last two inequalities in \eqref{P11}, we obtain the exponential decay estimate \eqref{energydecay}.
\end{proof}

\begin{Remark}\label{Plate}{\rm
As another example, we could consider the damped plate equation (see e.g. \cite{NicaisePignotti14}
for the model details). The analysis is analogous to the wave case above.
Then, under suitable assumptions, the exponential stability result holds for that model.}
\end{Remark}

\subsection{A damped elasticity system}
\bigskip

Let $\Omega\subset\RR^d,$ and let  $\tilde{\mathcal{O}},\mathcal O \subset \Omega $ be as in the previous example.
We consider the following elastodynamic   system
\begin{equation}\label{elasticity}
\begin{array}{l}
\displaystyle{u_{tt}(x,t)-\mu \Delta u(x,t)- (\lambda+\mu)\nabla \, \dv u+a\chi_{\mathcal{O}}(x) u_t(x,t)}\\
\vspace{2 pt}\hspace{3,5 cm}
\displaystyle{+k(t)\chi_{\tilde{\mathcal O}}(x)u_t(x,t-\tau(t))=0, \quad (x,t)\in\Omega\times (0,+\infty),}\\
\vspace{2 pt}\displaystyle{u(x,t)=0, \ \ \quad\quad (x,t)\in\partial\Omega\times (0,+\infty),}\\
\displaystyle{u(x, s)=u_0(x, s), \quad u_t(x,s)=u_1(x,s), \quad (x,s)\in {\Omega}\times [-\bar\tau,0],}
\end{array}
\end{equation}
where $a$ is a positive constant, $\tau(t)$ is the time delay function satisfying $0\le\tau(t)\le\bar\tau,$  and the delayed damping coefficient $k(\cdot):[-\bar\tau,+\infty)\to (0,+\infty)$ is a $\mathcal{L}^1_{loc}([-\bar\tau ,+\infty))$ function satisfying \eqref{damp_coeff}. 
Note that, in this case, the function $u$ is vector-valued and takes values in $\RR^d$ while $\lambda$ and $\mu$ are positive constants usually called Lam\'{e} coefficients.
%%$H=L^2(\Omega),$ and the operator $A: {\mathcal D}(A)\rightarrow H$ defined as
%%$A=-\Delta$ with domain ${\mathcal D}(A)=H^2(\Omega)\cap H_0^1(\Omega)$. 
%%The operator $A$ is positive, self-adjoint, with dense domain and compact inverse in $H.$
%%We then define $W_1:=L^2({\mathcal O}),  W_2:=L^2(\tilde{\mathcal O}),$ and the operators 
%%$$C:W_1\rightarrow H: \quad v\rightarrow \sqrt{a} \tilde v\chi_ {\mathcal O},$$
%%$$B:W_2\rightarrow H: \quad v\rightarrow \tilde v\chi_ {\tilde{\mathcal O}},$$
%where $\tilde v$
%is the extension of v by zero outside  ${\mathcal O}$ and ${\tilde{\mathcal O}}$ %respectively.
Denoting $v(t)=u_t(t)$ and $U(t)=(u(t),v(t))^T,$ for any $t\ge 0,$, we can rewrite system \eqref{elasticity} in the abstract form \eqref{PA}, with ${\mathcal H}={H_0^1(\Omega)}^d\times {L^2(\Omega)}^d$,
$$
 A=\begin{pmatrix}
0 & Id \\
\mu\Delta+ (\lambda+\mu)\nabla \dv & -a\chi_{\mathcal O}
\end{pmatrix}
$$
and 
$$
 B \begin{pmatrix} u \\ v \end{pmatrix} = \begin{pmatrix} 0 \\ -\chi_{\tilde{\mathcal{O}}} v \end{pmatrix},  \quad \forall\ t \geq 0.
$$
We know that ${A}$ generates an exponentially stable $C_0$-semigroup $\{S(t)\}_{t\geq 0}$ (see e.g. \cite{Valeria}), namely there exist $\omega,M>0$ such that
$$
||S(t)||_{\mathcal{L}(\mathcal H)} \leq Me^{-\omega t}, \quad \forall \ t\geq 0.
$$
Hence, under the assumption \eqref{assumption_delay}, the stability estimate of Theorem \ref{generaleCV} holds for such a model. Therefore, we can deduce an exponential decay estimate for the
energy functional
$$\begin{array}{l}
\displaystyle{{\cal E}(t):= \frac 12 \int_\Omega |u_t(x,t)|^2 dx+\frac 12 \int_{\Omega}  {\Big [\mu \sum_{i,j=1}^n\Big (\frac {\partial}{\partial x_i} u_j(x,t)\Big )^2+ (\lambda+\mu)\vert \dv u\vert \Big ]dx}}\\
\hspace{6,5 cm} \displaystyle{
+\frac 12 \int_{t-\bar\tau}^t\int_{\tilde{\mathcal O}} |k(s)|\cdot |u_t(x,s)|^2 dx ds.}
\end{array}
$$
\begin{Theorem}\label{exp_elasticity}
Assume \eqref{assumption_delay}. Then, for all initial data $(u_0, u_1)\in 
C([-\bar\tau, 0]; H_0^1(\Omega)^d\times L^2(\Omega)^d),$
the solution to
\eqref{elasticity} satisfies the energy decay estimate
\begin{equation}\label{energy_decay}
{\cal E}(t) \le \bar{C} e^{-\beta^* t}, \quad t\ge 0,
\end{equation}
where $\bar C$ is a constant depending on the initial data and $\beta^*>0.$
\end{Theorem}

\begin{proof} The proof comes analogously to the one of Theorem \ref{exp_wave}.
\end{proof}

\noindent {\bf Acknowledgements.} We would like to thank GNAMPA and UNIVAQ  for the support.

\end{document}